\documentclass{amsart}

\usepackage{amsmath,amssymb, amsopn, amsthm}
\usepackage{enumerate, framed, xcolor}

\theoremstyle{definition} 
\newtheorem{theorem}{Theorem}[section]

\newtheorem{proposition}[theorem]{Proposition}
\newtheorem{corollary}[theorem]{Corollary}
\newtheorem{claim}[theorem]{Claim}

\newtheorem{definition}[theorem]{Definition}

\renewenvironment{proof}[1][\unskip]{%
\par
\noindent
\textbf{Proof #1.}
\noindent}
{\hfill$\blacksquare$

\bigskip}

\let\<=\langle
\let\>=\rangle

\def\[#1\]{\begin{align*}#1\end{align*}}

\def\land{\wedge}       
\def\lor{\vee}

\def\lnot{\neg}

\let\models=\vDash

\def\defeq{\overset{\text{def}}{=}}

\DeclareMathOperator{\ran}{ran}

\def\gA{\mathfrak{A}}

\def\gB{\mathfrak{B}}

\makeatletter
\newcommand\RedeclareMathOperator{%
  \@ifstar{\def\rmo@s{m}\rmo@redeclare}{\def\rmo@s{o}\rmo@redeclare}%
}
\newcommand\rmo@redeclare[2]{%
  \begingroup \escapechar\m@ne\xdef\@gtempa{{\string#1}}\endgroup
  \expandafter\@ifundefined\@gtempa
     {\@latex@error{\noexpand#1undefined}\@ehc}%
     \relax
  \expandafter\rmo@declmathop\rmo@s{#1}{#2}}
\newcommand\rmo@declmathop[3]{%
  \DeclareRobustCommand{#2}{\qopname\newmcodes@#1{#3}}%
}
\@onlypreamble\RedeclareMathOperator
\makeatother

\DeclareMathOperator{\Sb}{Sb}
\DeclareMathOperator{\Rd}{Rd}
\DeclareMathOperator{\Eq}{Eq}
\DeclareMathOperator{\C}{C}
\RedeclareMathOperator{\S}{S}
\RedeclareMathOperator{\P}{P}

\def\RA{\mathsf{RA}}

\def\FPA{\mathsf{FPA}}
\def\QPA{\mathsf{QPA}}
\def\c{\mathsf{c}}
\def\s{\mathsf{s}}
\def\p{\mathsf{p}}
\def\ACSA{\qquad\Leftrightarrow\qquad}

\begin{document}

\title[Substitutions of variables are finitely axiomatizable]{Substitutions of variables are finitely axiomatizable over quantifications and permutations}

\author[Andr\'eka]{Hajnal Andr\'eka}

\address{%
Rényi Institute of Mathematics \\
Re\'altanoda utca 13-15 \\
1053 Budapest \\
Hungary
}

\email{andreka.hajnal@renyi.hu}


\author[Gyenis]{Zal\'an Gyenis}
\address{
Jagiellonian University\\
Grodzka 52 \\
	31-007 Krak\'ow \\
	Poland}

\email{zalan.gyenis@uj.edu.pl}

\author[N\'emeti]{Istv\'an N\'emeti}
\address{%
Rényi Institute of Mathematics \\
Re\'altanoda utca 13-15 \\
1053 Budapest \\
Hungary
}
\email{nemeti.istvan@renyi.hu}




\date{\today}
\dedicatory{Dedicated to the 60th birthday of Jean-Yves B\'eziau}

\begin{abstract}
This paper proves that the equational theory of the class $\RA_{\alpha}^{csp}$ of representable polyadic algebras 
is finitely axiomatizable over its substitution-free reduct $\RA_{\alpha}^{cp}$, for finite $\alpha$. That is, substitutions of variables in finite variable first-order logic can be described by finitely many axioms over the Boolean operations, existential quantifiers and permutations of variables.
\end{abstract}

\maketitle

\section{Mind your operations}

A significant line of research in algebraic logic concerns the (non-)finite axiomatizability of certain logical operations—more precisely, of classes of representable algebras—over others. Rather than attempting a comprehensive survey of known results (for which see the summaries in \cite{AAPALII97ec} and \cite{ANTnew}), we briefly highlight a few illustrative examples. We assume that the dimension is finite and $\ge 3$.
\begin{itemize}
	\item \cite{M69} and \cite{J69}: two classic results in the subject are that 
	the classes of representable cylindric and representable polyadic equality algebras are not finitely axiomatizable.
	
	\item \cite{ANTnew}: 
	The class of representable polyadic equality algebras 
	cannot be axiomatized by adding finitely many 
	equations to the equational theory of representable cylindric algebras of
	dimension $\alpha$. That is, the operations $\p_{ij}$ are not finitely 
	axiomatizable over $\cap, -, \c_i, \mathsf{d}_{ij}$. 

	\item \cite{SN96}: Finite axiomatization of substitutions 
	$\s_{ij}$ over the Boolean operations.
	
	\item \cite{M69}: The Boolean operations together with
	 cylindrifications and substitutions are not finitely axiomatizable, 
	 this proves nonfinite axiomatizability of cylindrifications 
	 over the Booleans and substitutions.	 
\end{itemize}

Additional results and open problems are summarized in  
\cite[Figure 1]{AAPALII97ec}. The open cases typically involve substitution operations in the presence of cylindrifications and absence of the diagonals. Substitution operations are important ones and well investigated, see, e.g.,  Pinter \cite{P1}, \cite{P2}.

One of open questions posed in \cite{AAPALII97ec} concerns the finite axiomatizability of substitutions over the Boolean operations together with cylindrifications, and permutations. In this paper, we solve this problem by showing that such an axiomatization is indeed possible (in finite dimension).
We show that the usual polyadic equations provide an explicit finite list of axioms, thereby settling the question in \cite{AAPALII97ec}. 
In the algebraic parlance, our result is that the class of representable finitary polyadic algebras is finitely axiomatizable over its substitution-free reduct. 
This will be a consequence of a theorem, stating that a  polyadic algebra of finite dimension is representable as soon as its nonpermutational-substitutions free reduct is representable. This shows that the nonpermutational substitution operations do not contribute to Johnson's result \cite{J69}  that the class of the representable polyadic algebras is nonfinite axiomatizable.  
Another consequence is that cylindrifications are not finitely axiomatizable over the Booleans and the permutations. This was indicated as an open problem in \cite[Figure 1]{AAPALII97ec}, but it got solved in the meantime by J. D. Monk \cite{M99}.\\

Let us recall the relevant definitions and background.\\

For an arbitrary ordinal%
\footnote{Throughout we use that an ordinal is the set of all smaller ordinals.} $\alpha$, a quasi-polyadic algebra $\QPA_{\alpha}$
of dimension $\alpha$ is an algebra of the form
\[
	\<A, \cdot, -, \c_{\Gamma}, \s_{\tau}:\; \tau\in{}^{\alpha}\alpha
	\text{ is a finite transformation}, \Gamma\text{ is a finite subset of }\alpha\>
\]
satisfying certain axioms that we do not recall here, but only refer to \cite[Def.5.4.1]{HenkinMonkTarski1971}. The set of equations true in $\QPA_{\alpha}$ is not closed under permutations of the set of indices of the operations ($\QPA_{\alpha}$ is not Monk-type schema axiomatizable). However, there
exists an alternative definition containing schemas of equations closed under
permutations of the indices of the operations. To this end, one had to alter the similarity type. The so obtained algebras are called \emph{finitary polyadic algebras}, $\FPA_{\alpha}$, and by \cite[Theorem 1]{sain-thompson1991}, $\FPA_{\alpha}$ and $\QPA_{\alpha}$ are term-definitionally equivalent (for $\alpha>2$). 
In this paper we will use quasi-polyadic algebras in the form of having similarity type $\cdot, -, \c_i, \s_{ij}, \p_{ij}$ for $i,j<\alpha$, that is, in place of the traditional $\QPA$ we will use its term-definitional equivalent form $\FPA$ in \cite{sain-thompson1991}. We recall the definition below.

\begin{definition}
	By a \emph{finitary polyadic algebra} of dimension $\alpha$,
	$\FPA_{\alpha}$, we mean an algebra 
	$\gA = \<A, \cdot, -, \c_i, \s_{ij}, \p_{ij}\>_{i,j<\alpha}$ in
	which equations (F0-F9) below are valid for every $i, j, k < \alpha$.
	\begin{itemize}
		\item[(F0)] $\<A, \cdot, 0\>$ is a Boolean algebra\footnote{The derived Boolean operations $+$, $0$ etc. are taken as defined in the usual way.}, $\s_{ii}=\p_{ii}=id$, $\p_{ij}=\p_{ji}$.
		\item[(F1)] $x\leq \c_ix$
		\item[(F2)] $\c_i(x+y)=\c_ix+\c_iy$
		\item[(F3)] $\s_{ij}\c_ix = \c_ix$
		\item[(F4)] $\c_i\s_{ij}x = \s_{ij}x$ if $i\neq j$
		\item[(F5)] $\s_{ij}\c_kx = \c_k\s_{ij}x$ if $k\notin \{i,j\}$
		\item[(F6)] $\s_{ij}$ and $\p_{ij}$ are Boolean endomorphisms
		\item[(F7)] $\p_{ij}\p_{ij}x=x$
		\item[(F8)] $\p_{ij}\p_{ik}x=\p_{jk}\p_{ij}x$ if $i,j,k$ are all distinct
		\item[(F9)] $\p_{ij}\s_{ji}x=\s_{ij}x$
	\end{itemize}
\end{definition}

The set of all $U$-termed $\alpha$-sequences is denoted by ${}^\alpha U$, thus ${}^\alpha U$ is the set of all functions from $\alpha$ to $U$. For a sequence $s\in {}^{\alpha}U$, $i <\alpha$
and $u\in U$ we denote by $s(i/u)$ the sequence we obtain from $s$ by replacing its $i$th value with $u$. An $\alpha$-ary relation over $U$ is a subset $R\subseteq {}^{\alpha}U$.  Define the following unary operations
on $\alpha$-ary relations over $U$:
\begin{align*}
	\C_i^U(R) &= \{ s\in {}^{\alpha}U:\; s(i/u)\in R\text{ for some } u\in U\},\\
		\S_{ij}^U(R) &= \{ s\in {}^{\alpha}U :\; s(i/s_j)\in R \},\\
	\P_{ij}(R) &= \big\{ s :\; s(i/s_j)(j/s_i)\in R	\big\}.
\end{align*}
Note that $\P_{ij}(R)$ does not depend on $U$, while $\C_i^U(R)$ and $\S_{ij}^U(R)$ do depend on $U$. We often omit the superscript $U$ even in $\C_i$ and $\S_{ij}$ when this is not likely to lead to confusion. The set of all subsets of a set $V$ is denoted by $\Sb(V)$, so $\Sb({}^\alpha U)$ is the set of all $\alpha$-ary relations on $U$. The $\alpha$-dimensional
full finitary polyadic set algebra is the structure
\[
	\gA = \< \Sb({}^{\alpha}U), \cap, -, \C_i^U, \S_{ij}^U, \P_{ij}\>_{i,j<\alpha}\,.
\]
It is routine to check that full finitary polyadic set algebras satisfy (F0-F9), so they belong to the class $\FPA_{\alpha}$. Subalgebras of full algebras are called set algebras, and subalgebras isomorphic to products of full finitary polyadic set algebras
are called representable $\FPA_{\alpha}$'s. It was shown in \cite[Theorem 2(i)]{sain-thompson1991} that for $\alpha>2$ the class of representable $\FPA_{\alpha}$'s (or equivalently, $\QPA_{\alpha}$) 
is not definable by finitely many schemata of equations. 

For a class $K$ of algebras, $\mathbf{I}K,\mathbf{S}K,\mathbf{P}K$ denote the classes of all isomorphic copies, all subalgebras, and all products of elements of $K$, respectively. 
We will focus on the following classes of algebras
\begin{align*}
	\RA_{\alpha}^{csp} &= \mathbf{SI}\big\{ \< \Sb({}^{\alpha}U), \cap, -, \C_i^U, \S_{ij}^U, \P_{ij}^U\>_{i,j<\alpha}:\; U\text{ is a set} \big\}, \\
	\RA_{\alpha}^{cp} &= \mathbf{SI}\big\{ \< \Sb({}^{\alpha}U), \cap, -, \C_i^U, \P_{ij}^U\>_{i,j<\alpha}:\; U\text{ is a set} \big\}, \\
	\RA_{\alpha}^{p} &= \mathbf{SI}\big\{ \< \Sb({}^{\alpha}U), \cap, -, \P_{ij}^U\>_{i,j<\alpha}:\; U\text{ is a set} \big\}. 
\end{align*}
The superscripts \emph{csp}, \emph{cp} and \emph{p} refer to the similarity types of the algebras: 
\begin{align*}
	p &= \text{Boolean} + \{\p_{ij}:\; i,j<\alpha\}, \\ 
	cp &= \text{Boolean} + \{\c_i, \p_{ij}:\; i,j<\alpha\}, \\
	csp &= \text{Boolean} + \{\c_i, \s_{ij}, \p_{ij}:\; i,j<\alpha\}.	
\end{align*}
If $\gA$ is a $csp$-type algebra, then $\Rd_{cp}\gA$ denotes the $cp$-type reduct
of $\gA$, etc.

The goal of this paper is to prove that $\RA_{\alpha}^{csp}$ is finitely axiomatizable over $\RA_{\alpha}^{cp}$. 
What does finite axiomatizability of a class of algebras over another one mean?
Let $K$ and $L$ be classes of algebras such that the similarity type
of $K$ is contained in that of $L$. Let $\Eq(K)$ and $\Eq(L)$ denote the set of all
equations valid in $K$ and $L$ respectively. We say that 
\emph{$L$ is finitely axiomatizable over $K$} if $\Eq(K)\cup \Sigma$ is an
equational axiomatization of $\Eq(L)$ for some finite $\Sigma$. 
In our case $K = \RA_{\alpha}^{cp}$ consists of all appropriate subreducts
of elements of $L=\RA_{\alpha}^{csp}$, and thus finite axiomatizability of 
$\RA_{\alpha}^{csp}$ over $\RA_{\alpha}^{cp}$ means that 
$\Eq(\RA_{\alpha}^{csp})$ has an axiomatization in which the operation symbols not present in the type $cp$ occur finitely many times. 

We will derive the above theorem as corollary to our result (Theorem \ref{thm:1a}) stating that a finitary polyadic equality algebra of finite dimension is representable iff its substitutions-free reduct is representable.

\section{Finite axiomatizability}

The class of finite-dimensional representable polyadic algebras is not finitely axiomatizable, see \cite{J69}. The following theorem says that the non-permutational substitution operations do not contribute to this non-finite axiomatizability result.

\begin{theorem}\label{thm:1a}
		A finitary polyadic algebra of finite dimension is representable iff its substitution-free reduct is representable.
\end{theorem}
\begin{proof} Let $\alpha$ be a finite ordinal. The substitution-free reduct of a representable $\FPA_\alpha$ is representable, since if $\gA\subseteq \Pi_{i\in I}\gB_i$ then  $\Rd_{cp}\gA\subseteq \Pi_{i\in I}\Rd_{cp}\gB_i$. 
	
In the other direction, first we show that it is enough to concentrate on set algebras, i.e., it is enough to prove that if the reduct of an $\FPA_{\alpha}$ is isomorphic to a set algebra, then the algebra itself is isomorphic to a set algebra. Indeed, assume that we have already proved that
\begin{align}
	\mbox{if $\Rd_{cp}\gA\in\RA^{cp}$ and $\gA\in\FPA_\alpha$, then $\gA\in\RA^{csp}$.}
	\label{*}
\end{align}	
Take an arbitrary $csp$-type algebra 
$\gA  = \<A, \cdot, -, \c_i, \s_{ij}, \p_{ij}\>_{i,j<\alpha}$. Assume
that $\gA\models $(F0-F9), and $\Rd_{cp}\gA$ is representable, say 
\[
f: \<A, \cdot, -, \c_i, \p_{ij}\>_{i,j<\alpha} \quad\rightarrowtail\quad
\Pi_{k\in K}\<\Sb({}^{\alpha}U_k), \cap, -, \C_i, \P_{ij}\>_{i,j<\alpha}
\]	
holds for some $f$ and $U_k$, $k\in K$. 
Let $k\in K$, let $\pi_k$ be the projection function, and let $g:\Rd_{cp}\gA\rightarrow \<\Sb({}^{\alpha}U_k), \cap, -, \C_i, \P_{ij}\>_{i,j<\alpha}$ be the composition of $f$ with $\pi_k$. Let $\gB_k$ be the $g$-image of $\gA$ and let $i,j<\alpha$. We show that (F0-F9) imply that the $g$-image of the function $\s_{ij}$ on $A$ is a function on $B_k$.  To this end, we have to show that if $g(a)=g(b)$ then $g(\s_{ij}a)=g(\s_{ij}b)$,  for any $a,b\in A$. Note that $g$ is a Boolean homomorphism that preserves also the cylindrifications $\c_i$. Now, $g(a)=g(b)$ iff,  by $g$ being a Boolean homomorphism, $g(a\oplus b)=0$ where $\oplus$ denotes the Boolean symmetric difference. Similarly, $g(\s_{ij}a)=g(\s_{ij}b)$ iff $g(s_{ij}a\oplus \s_{ij}b) = g(\s_{ij}(a\oplus b))=0$, by (F6). Therefore, it is enough to show that $g(\s_{ij}x)=0$ whenever $g(x)=0$. Indeed,
\bigskip

\noindent
$g(x)=0$\qquad implies, by $g$ being a $\c_i$-homomorphism,\\
$g(\c_ix)=\c_i0=0$,\qquad implies, by (F3),\\
$g(\s_{ij}\c_ix)=0$, \qquad implies, by $\s_{ij}x\le\s_{ij}\c_ix$ and Boolean homomorphism,\\
$g(\s_{ij}x)\le g(\s_{ij}\c_ix)=0$,\qquad so\\
$g(\s_{ij}x)=0$.
\bigskip

\noindent
We have seen that the $g$-images of the (abstract) substitution operations of $\gA$ are functions on $\gB_k$. Let $\s_{ij}^{+}$ denote these projected functions and let $\gB_k^+ = \< \gB_k, \s_{ij}^{+}\>_{i,j<\alpha}$, for all $k\in K$. Then $g:\gA\to\gB_k^+$ and so $\gB_k^+\in\FPA_\alpha$ by $\gA\in\FPA_\alpha$. Also $\Rd_{cp}\gB_k^+ = \gB_k\in\RA_{\alpha}^{cp}$ by its construction. Thus, $\gB_k^+\in\RA_{\alpha}^{csp}$, by \eqref{*} and so $\gA\in\mathbf{SIP}\RA_{\alpha}^{csp}$ since $f:\gA\rightarrowtail \Pi_{k\in K}\gB_k^+$. That is, $\gA$ is representable. We have seen that it is enough to prove \eqref{*}.
\bigskip

To prove \eqref{*}, take an arbitrary $csp$-type algebra 
$\gA  = \<A, \cdot, -, \c_i, \s_{ij}, \p_{ij}\>_{i,j<\alpha}$. Assume
that $\gA\models$ (F0-F9), and 
\[
f: \<A, \cdot, -, \c_i, \p_{ij}\>_{i,j<\alpha} \quad\rightarrowtail\quad
\<\Sb({}^{\alpha}U), \cap, -, \C_i, \P_{ij}\>_{i,j<\alpha}
\]	
for some $f$ and $U$.
	
	First we show that the substitution-free reduct 
	has a representation on the repetition-free sequences in some
	${}^{\alpha}V$, and then we will use the (abstract) $\s_{ij}$ operations of $\gA$ to 
	tell which elements of $A$ the sequences with repetitions in
	${}^{\alpha}V$ should be mapped into.
	
	Let $H$ be a set with cardinality greater than $\alpha$, i.e., $|H|>|\alpha|$ and let $V=U\times H$.
	Let us write $W$ for the repetition-free sequences in
	${}^{\alpha}V$, that is
	\[
		W \defeq \big\{ s\in {}^{\alpha}(U\times H):\; 
			s_i\neq s_j\text{ for all } i<j<\alpha	\big\}\,.
	\]
	For $s\in {}^{\alpha}(U\times H)$ we write 
	\[
		\hat{s} = \< s_i(0):\; i<\alpha\> \in {}^{\alpha}U\,.
	\]
	For example, for $s = \< \<u,i\>, \<v,j\>, \ldots, \<w,k\>\>$ we have
	$\hat{s} = \<u, v, \ldots, w\>$. 
	For $X\subseteq W$ and $i\in\alpha$ let $\C_i^W(x)=W\cap\C_i^V(X)$ and ${}_{W}\!-(X)=W-X$. Define 
	\[
		g:\<A, \cdot, -, \c_i, \p_{ij}\>_{i,j<\alpha} \quad\rightarrowtail\quad
		\<\Sb(W), \cap, {}_{W}\!-, \C_i^W, \P_{ij}\>_{i,j<\alpha}
	\]
	by letting
	\[
		g(x)\defeq \big\{  s\in W :\; \hat{s}\in f(x)\big\}
	\]
	for all $x\in A$.
	\begin{claim}\label{claim:1}
		$g$ is a homomorphism.
	\end{claim}
	\begin{proof}[of Claim \ref{claim:1}] Let $s\in W$ be arbitrary.
		We check first the Boolean operation $g(x\cdot y) = g(x)\cap g(y)$: 
		\begin{align*}
			s \in g(x\cdot y) 
			&\ACSA \hat{s}\in f(x\cdot y)  \tag{by def. of $g$} \\ 
			&\ACSA \hat{s}\in f(x)\cap f(y) \tag{since $f$ is a homomorphism}\\
			&\ACSA \hat{s}\in f(x)\text{ and }\hat{s}\in f(y) \tag{by def. of $\cap$}\\
			&\ACSA s \in g(x)\cap g(y) \tag{by def. of $g$ and $\cap$}
		\end{align*}
		As for the complementation $g(-x)={}_{W}\!-g(x)$, we have 
		\begin{align*}
			s \in g(-x) 
			&\ACSA \hat{s} \in f(-x) \tag{by def. of $g$}\\
			&\ACSA \hat{s} \notin f(x)  \tag{since $f$ is a homomorphism}\\
			&\ACSA s\notin g(x) \tag{by def. of $g$} \\
			&\ACSA s \in W-g(x) = {}_{W}\!-g(x) \tag{by def of ${}_{W}\!-$}
		\end{align*}
		Next, we show $g(\c_kx) = \C_k^Wg(x)$ for $k<\alpha$. Let $a\in H - \{ s_i(1) : i<\alpha\}$. There is such an $a$ by $|H|>|\alpha|$. Note that if $s_i=\langle u,b\rangle$ then $s_i(0)=u$ and $s_i(1)=b$.
		\begin{align*}	
			s \in g(\c_kx) 
			&\ACSA  \hat{s}\in f(\c_kx)  \tag{by def. of $g$}\\
			&\ACSA \hat{s}\in \C_k^Uf(x) \tag{since $f$ is a homomorphism}\\
			&\ACSA (\exists u\in U)\  \hat{s}(k/u)\in f(x)  \tag{by def.\ of $\C_k^U$}\\
			&\ACSA (\exists u\in U)\  s(k/\<u,a\>)\in g(x) \tag{by def.\ of \;$\hat{}$, $g$ and $a$} \\
			&\ACSA s \in \C_k^{W}g(x) \tag{by def. of $\C_k^W$}
		\end{align*}
		Finally, for $g(\p_{ij}x) = \P_{ij}g(x)$ we have
		\begin{align*}
			s\in g(\p_{ij}x) 
			&\ACSA \hat{s} \in f(\p_{ij}x) \tag{by def. of $g$}\\
			&\ACSA \hat{s}(i/\hat{s}_j)(j/\hat{s}_i) \in f(x) \tag{since $f$ is a homomorphims}\\
			&\ACSA s(i/{s}_j)(j/{s}_i)\in g(x) \tag{by def. of \;$\hat{}$ and $g$}\\
			&\ACSA s \in \P_{ij}g(x) \tag{by def. of $\P_{ij}$}
		\end{align*}
	\end{proof}
	
		By a finite transformation on $\alpha$ we mean a function $\sigma:\alpha\to\alpha$ which moves only finitely many elements of $\alpha$. Let $T$, $P$ and $S$ denote the sets of all finite transformations on $\alpha$, 
	all finite permutations of $\alpha$, and all non-permutational finite transformations on
	$\alpha$, respectively, i.e.
	\[
	T = \{\sigma\in{}^{\alpha}\alpha :  |\{ i\in\alpha : \sigma(i)\ne i\}|<\omega\}, 
	P = \{f\in T:\; \ran(f)=\alpha\},
	S = T-P.
	\]
	The (semigroup of) finite transformations can be generated by the operations
	$[i,j]$ and $[i/j]$ for $i,j<\alpha$, where $[i,j]$ denotes the function that exchanges $i$ and $j$, and leaves every other element of $\alpha$ fixed; while $[i/j]$ is the function that maps $i$ onto $j$, and leaves other elements of $\alpha$ fixed. $[i,j]$ is called the \emph{transposition} of $i$ and $j$ and $[i/j]$ is called the \emph{replacement} of $i$ by $j$. It is known that each $\sigma\in P$ can be written as a composition of transpositions and each $\sigma\in S$ can be written as a composition of replacements. Let $\circ$ denote usual function-composition, i.e.,
	\[ (\sigma\circ\tau)(i) = \sigma(\tau(i))\qquad\mbox{for all }i<\alpha.\] 
It is proved in \cite[Theorem 1(i)]{sain-thompson1991} that by the use of (F0-F9),  one can define 
\begin{align}
	\s_{\sigma}(x) = \s_{i_1,j_1}\cdots\s_{i_n,j_n}(x)\quad
		&\text{ where }\quad \sigma = [i_1/ j_1]\circ\cdots\circ [i_n/ j_n],\\
	\s_{\sigma}(x) = \p_{i_1,j_1}\cdots\p_{i_n,j_n}(x)\quad
	&\text{ where }\quad \sigma = [i_1,j_1]\circ\cdots\circ [i_n,j_n].
\end{align}
For these defined terms, the following usual polyadic equations (S1-S6) are true in  $\FPA_\alpha$ (for proof see the proof of \cite[Theorem 1(i)]{sain-thompson1991}):
\bigskip

\begin{itemize}
	\item[(S1)] $\s_{\sigma}x = \c_k\s_{\sigma}x$ if $k\notin\ran(\sigma)$.
	\item[(S2)] $\s_{\sigma}\c_ix = \s_{\delta}\c_ix$ if $\sigma$ and $\delta$ differ only at $i$.
	\item[(S3)] $\s_{\sigma}\c_ix = \c_k\s_{\sigma}x$ if 
	$\{j:\; \sigma(j)=k\} = \{i\}$.
	\item[(S4)] $\s_{\sigma}$ is a Boolean homomorphism.
	\item[(S5)] $\s_{\sigma}\s_{\eta}x = \s_{\sigma\circ\eta}x$.
	\item[(S6)] $\s_{ij}x = \c_i\s_{ij}x$.			
\end{itemize}
\bigskip

\noindent Finally, we define $\S_{\sigma}^U:\Sb({}^{\alpha}U)\to \Sb({}^{\alpha}U)$ by
\[ \S_{\sigma}^U(X) = \{ s\in{}^{\alpha}U : s\circ\sigma\in X\}\quad  \text{ for } X\subseteq{}^{\alpha}U.\]
In the above, $s\circ\sigma$ is the sequence $s$ rearranged along $\sigma$, i.e., 
\[ s\circ\sigma = \langle s_{\sigma0}, s_{\sigma1},\dots \rangle.\]
Then  $\S_{ij}^U=\S_{[i/j]}^U$ and $\P_{ij}^U=\S_{[i,j]}^U$.	
\bigskip	
	
We are ready to define our function $h:A\to\Sb({}^\alpha V)$; the  already defined function $g:A\to\Sb(W)$ guides us in this, as follows. We want $g(x)\subseteq h(x)$ and  we want $h$ to be a homomorphism also w.r.t.\ the substitution functions $\s_{ij}$. Assume $z\in g(w)$ and $w\le \s_{ij}y$. Then we need $z\in g(w)\subseteq \S_{ij}h(y)$, so $z(i/z_j)\in h(y)$ by the definition of $\S_{ij}$. If also $y\le\s_{kl}x$, then we need $z(i/z_j)(k/z_l)\in h(x)$. This leads to the following definition.
	\[
		h(x) = g(x)\cup\big\{
			z\circ\sigma:\; z\in g(\s_{\sigma}x), \ \sigma\in S
		\big\}.
	\]
Claim \ref{claim:1} implies that
	\begin{align}
	g(\s_{\sigma}x) = \S_{\sigma}^Wg(x) \quad\text{ for every bijection } \sigma\in P.
	\label{g3}
	\end{align}
	By using \eqref{g3}, one can see that
	\begin{align}
		h(x) = \big\{ z\circ \sigma:\; z\in g(\s_{\sigma}x), \ \sigma\in T\big\} \text{ and }g(x) = h(x)\cap W.
		\label{g4}
	\end{align}
	We will use the function $h$ in the form as in \eqref{g4}, and we will also use that by Claim \ref{claim:1} we have
	\begin{align}
		g(\c_ix) = \C_i^Vg(x)\;\cap W\label{g2}
	\end{align}
	for all $x\in A$ and $i<\alpha$. In what follows we prove that
	
		\[
		h: \<A, \cdot, -, \c_i, \p_{ij}, \s_{ij}\>_{i,j<\alpha}
		 \quad\rightarrowtail\quad
		\<\Sb({}^{\alpha} V), \cap, -, \C_i^V, \P_{ij}, \S_{ij}^V\>_{i,j<\alpha}
	\]
	\\
	\noindent
	is a homomorphism.

	Homomorphism w.r.t.\ the Booleans $\cdot,-$ follows from $g$ and $\s_{\sigma}$ respecting the Booleans by Claim \ref{claim:1} and (F6); we will need also $\alpha<\omega$ in the case of complementation. 
	
	As for $\cdot$ we have
	\begin{align*}
		h(x\cdot y) &\ =\ \{ z\circ\sigma : z\in g(\s_{\sigma}(x\cdot y)), \sigma\in T\} \\
		  &\ =\ \{ z\circ\sigma :\; z\in g(\s_{\sigma}x\cdot \s_{\sigma}y), \sigma\in T\} \\
		  &\ =\ \{ z\circ\sigma :\; z\in (g(\s_{\sigma}x)\cap g(\s_{\sigma}y)), \sigma\in T\} \\
		  &\ =\ \{ z\circ\sigma :\; z\in g(\s_{\sigma}x), \sigma\in T\}\cap \{ z\circ\sigma :\; z\in g(\s_{\sigma}y), \sigma\in T\} \\
		  &\ =\ h(x)\cap h(y)\,.
	\end{align*}
	
	\noindent
	For checking complementation $-$, notice first that each sequence $s\in{}^\alpha U$ is of form $z\circ\sigma$ for some $z\in W$ and $\sigma\in T$, because $\alpha$ is finite and $|H|\ge\alpha$. I.e.,
	\begin{align}
		{}^\alpha V = \{ z\circ\sigma :\;  z\in W\mbox{ and }\sigma\in T\}.
		\label{h-}
	\end{align}

	\noindent Then we have 
	\begin{align*}
		 {}^\alpha V - h(x) &\ =\ \{ z\circ\sigma : z\circ\sigma\notin h(x),\; z\in W,\; \sigma\in T\} \\
		 &\ \subseteq\ \{ z\circ\sigma : z\notin g(\s_{\sigma}(x)),\; z\in W,\; \sigma\in T\} \\
		 &\ \subseteq\ \{ z\circ\sigma : z\in g(-\s_{\sigma}x),\; \sigma\in T\} \\
		 &\ \subseteq\ \{ z\circ\sigma : z\in g(\s_{\sigma}(-x)),\; \sigma\in T\} \\
		 &\ \subseteq\  h(-x)\,.
	\end{align*}
\noindent
For the other direction, let $s\in h(-x)$ be arbitrary. Assume that $s\in h(x)$, we derive a contradiction. Now, $s\in h(-x)\cap h(x)=h(-x\cdot x)=h(0)$, so $s=z\circ \sigma$ for some $z\in g(\s_\sigma 0)=g(0)=\emptyset$ by the definition of $h$, (S4) and $g$ being a Boolean homomorphism. Thus, $s\notin h(x)$ showing $h(-x)\subseteq{}^\alpha V-(x)$.

	In the rest of the proof, we will write $\sigma z$ in place of $z\circ\sigma$. $\C_i$ and
	$\S_{\sigma}$ without the superscripts denote $\C_i^V$ and
	$\S_{\sigma}^V$, respectively. The axioms (S1-S6) will 
	be used in the proof, they follow from (F0-F9) by term-definitional equivalence of $\FPA$ and $\QPA$, 
	see \cite[Theorem 1(i)]{sain-thompson1991}.
	Next we show that $h$ is a homomorphism w.r.t.\ the cylindrifications $\c_i$.
	\begin{claim}\label{claim:2}
		$h(\c_ix) = \C_i^Vh(x)$\quad for $i<\alpha$ and $x\in A$.
	\end{claim}
	\begin{proof}[of Claim \ref{claim:2}]
		
		First we show $h(\c_ix)\subseteq \C_ih(x)$. Assume $w\in h(\c_ix)$. 
		This means, by definition of $h$, that 
		\begin{align}
			w=\sigma z\mbox{ for 
				some }z\in g(\s_{\sigma}\c_ix), \sigma\in T.
			\label{*1}
		\end{align} 
		
		Assume first that $\sigma$ is a bijection ($\sigma\in P$). Then 
		$z\in g(\s_{\sigma}\c_ix) = \S_{\sigma}g(\c_ix)$ by \eqref{*1} and \eqref{g3}. 
		Thus $w=\sigma z\in g(\c_i x)$ by definition of $\S_{\sigma}$. Hence, 
		$w\in \C_ig(x)\subseteq \C_ih(x)$ by \eqref{g2}, \eqref{g3} and
		since $w$ is repetition-free by $w=\sigma z$, \eqref{*1}, $\ran(g)\subseteq W$ and $\sigma\in P$. We have seen that
		$w\in \C_ih(x)$.
		
		Assume next that $\sigma$ is not a bijection. Then there is
		$k\notin\ran(\sigma)$ as $\sigma$ is finite. 
		Let $\delta = \sigma(i/k)$. Then 
		$\s_{\sigma}\c_ix =\s_{\delta}\c_ix = \c_k\s_{\delta}x$ by (S2) 
		and (S3). Thus, $z\in g(\s_{\sigma}c_ix) = g(c_k\s_{\delta}x) = C_kg(\s_{\delta}x)\cap W$ by \eqref{*1} and \eqref{g3}, so $z(k/u)\in g(\s_{\delta}x)$ for some $u$.
		Now,
		\begin{align} 
			w(i/u)=\delta(z(k/u))
			\label{w}
		\end{align}
		holds because of the following. Let $j\ne i$.  Then
		\begin{align*}
			w(i/u)(j) &= w(j)  \tag{by $j\ne i$} \\
			 &= (\sigma z)j \tag{by \eqref{*1}} \\
			 &= z(\sigma j) \tag{by definition} \\
			 &= z(\delta j) \tag{by $\delta=\sigma(i/k),\ j\ne i$} \\
			 &= z(k/u)(\delta(j)) \tag{by $k\ne\sigma(j) = \delta(j)$} \\
			 &= (\delta z(k/u))(j) \tag{by definition}
		\end{align*}
		
		\noindent For $j=i$ we have  $w(i/u)(i)= u$  and
		\begin{align*}
			\delta(z(k/u))(i) &= z(k/u)(\delta(i)) \tag{by definition} \\
			&= z(k/u)(k) = u \tag{by $\delta(i)=k$}
		\end{align*}
		
		\noindent 
		It follows that $w(i/u)\in h(x)$ by \eqref{w}, $z(k/u)\in g(\s_{\delta}x)$ and the definition of $h$. So $w\in \C_ih(x)$. \\
		
		Next we show \begin{align}
			C_ih(\c_ix)\subseteq h(\c_ix).
			\label{c}
		\end{align} To prove \eqref{c}, assume $w\in h(\c_ix)$ and $u\in V$, we show that $w(i/u)\in h(\c_ix)$. By $w\in h(\c_ix)$ and the definition of $h$, we have \eqref{*1}. We proceed by distinction of cases.
		
		Assume that $u\in \ran w$. If $w_i=u$ then $w=w(i/u)\in h(\c_ix)\subseteq C_ih(\c_ix)$ and we are done. Assume that $w_j=u$ for $j\ne i$. Let $\delta=\sigma(i/j)$. Then $w(i/u)=\delta z$ and $z\in g(\s_\sigma\c_ix)=g(\s_\delta\c_ix)$ by \eqref{*1} and (S2). So $w(i/u)\in h(\c_ix)$ by the definition of $h$.
		
		Assume now that $u\notin\ran w$  and $\sigma\in P$. Then $w(i/u)$ is repetition-free since $w=\sigma z$  and $z\in W$ by \eqref{*1}. Now, $w(i/u)\in C_ig(\c_ix)\cap W =  g(\c_i\c_ix) = g(\c_ix)\subseteq h(\c_ix)$ by \eqref{g2}, $\c_i\c_ix=\c_ix$ and \eqref{g4}.
		
		Assume now that $u\notin\ran(w)$ and $\sigma\notin P$. Assume first that $u\in\ran(z)$. Then $u=z_k$ for some $k\notin\ran \sigma$ by $u\notin\ran w$ and \eqref{*1}. Let $\delta=\sigma(i/k)$.  Then $w(i/u)=\delta z$ and $z\in g(\s_\sigma\c_ix) = g(\s_\delta\c_ix)$  by (S2), and so $w(i/u)\in h(\c_ix)$ by the definition of $h$. Assume next that $u\notin\ran z$. There is $k\notin\ran\sigma$ by $\sigma\notin P$. Let $\delta=\sigma(i/k)$. Then $w(i/u)=\delta z(k/u)$ and $z(k/u)\in W$.  So, $z(k/u)\in C_kg(\s_{\sigma}\c_ix)\cap W = g(\c_k\s_\sigma\c_ix) = g(\s_{\sigma}\c_ix) = g(\s_{\delta}\c_ix)$ by \eqref{*1}, \eqref{g2} and (S1), (S2). Thus, $w(i/u)\in h(\c_ix)$ by the definition of $h$ and we are done with showing \eqref{c}.
		\bigskip
		
		Now, $\C_ih(x)\subseteq h(\c_ix)$ follows easily from \eqref{c}: $C_ih(x)\subseteq C_ih(\c_ix) = h(\c_ix)$.	This completes the proof of Claim \ref{claim:2}.
	\end{proof}
	
	We show homomorphism of $h$ w.r.t.\ $\s_{ij}$ and $\p_{ij}$ for $i,j\in\alpha$ by showing homomorphism of $h$ w.r.t. $\s_{\eta}$ for all $\eta\in T$.
	
	\begin{claim}\label{claim:3}
		$h(\s_{\eta}x) = \S_{\eta}h(x)$ for all $\eta\in T$ and $x\in A$.
	\end{claim}
	\begin{proof}[of Claim \ref{claim:3}]
		First we show $h(\s_{\eta}x)\subseteq \S_{\eta}h(x)$. Assume that
		$w\in h(\s_{\eta}x)$. Then $w=z\circ\sigma$ for some 
		$z\in g(\s_{\sigma}\s_{\eta}x)$. Thus, 
		$w\circ\eta = z\circ\sigma\circ\eta$, and $z\in g(\s_{\sigma\circ\eta}x)$
		by (S5). So, $w\circ \eta\in h(x)$ by the definition of $h$, and so
		$w\in \S_{\eta}h(x)$ by the definition of $\S_{\eta}$. 
		
		For the other inclusion, it is enough to show 
		$\S_{\eta}h(x)\subseteq h(\s_{\eta}x)$ separately for $\eta = [i/j]$
		and $\eta = [i,j]$, for all $i\neq j<\alpha$. 
		
		Assume $w\in S_{ij}h(x)$, we want to show $w\in h(\s_{ij}x)$. We have $w
		\circ [i/j]\in h(x)$ by the definition of $S_{ij}$. Then
		$w\circ [i/j]=z\circ\sigma$ for some
		$z\in g(\s_{\sigma}x)$ and $\sigma\in T$, by the definition of $h$. We have
		$z\circ\sigma=w\circ [i/j] = w\circ[i/j]\circ[i/j]=z\circ\sigma\circ[i/j]$ and so
		$\sigma=\sigma\circ[i/j]$ by $[i/j]\circ[i/j] = [i/j]$ and
		$z$ being repetition-free. Thus, $w\circ [i/j]=z\circ \sigma$
		and $z\in g(\s_{\sigma\circ[i/j]}x) = g(\s_{\sigma}\s_{ij}x)$, 
		and so $w\circ[i/j]\in h(\s_{ij}x)$ by the definition of $h$. Now
		we use that $w$ and $w\circ[i/j]$ differ only at $i$ and we have already
		seen that $h$ is a homomorphism for $\c_i$. Therefore, by (S6) we get
		\[
		w\in \C_i\{ w\circ [i/j]\} \subseteq 
		\C_i h(\s_{ij}x) = h(\c_{i}\s_{ij}x)
		=h(\s_{ij}x).
		\]
		
		Similarly, $w\circ[i,j]\in h(x)$ implies that $w\circ[i,j]=z\circ\sigma$
		for some $z\in g(\s_{\sigma}x)$. Then $w=z\circ\sigma\circ[i,j]$
		and $z\in g(\s_{\sigma}\p_{ij}\p_{ij}x) =
		g(\s_{\sigma\circ[i,j]}\p_{ij}x)$ by $[i,j]\circ[i,j]$ being 
		the identity, (S5) and the definition of $\s_{\sigma}$. So, 
		$w\in h(\p_{ij}x)$ by the definition of $h$.
	\end{proof}
	
	Summing it up, $h$ provides a representation of $\gA$, and this completes
	the proof of Theorem \ref{thm:1a}.
	
	\end{proof}

Recall that $T,P$ and $S$ denote the sets of all finite transformations on $\alpha$, all permutational elements of $T$ and all non-permutational elements of $T$, respectively. Assume that
\[
\<A, \cdot, -, \c_{\Gamma}, \s_{\tau}:\; \tau\in T,\  \Gamma\text{ is a finite subset of }\alpha\>
\]
is a $\QPA_{\alpha}$. By its nonpermutational-substitutions free reduct we mean the algebra
\[
\<A, \cdot, -, \c_{\Gamma}, \s_{\tau}:\; \tau\in P,\ \Gamma\text{ is a finite subset of }\alpha\>.
\]
The following corollary is an immediately consequence of Theorem \ref{thm:1a} and the proof of \cite[Theorem 1(i)]{sain-thompson1991}.
\begin{corollary}\label{cor:1}
	A $\QPA_{\alpha}$ is representable iff its nonpermutational-substitutions free reduct is representable, for finite $\alpha$.
\end{corollary}

As another corollary to Theorem \ref{thm:1a}, we can fill in the statuses of two open lines in \cite[Figure 1]{AAPALII97ec}. We note that the status of the second line was already known, it is a result of J.\ D.\ Monk, see \cite{M99}.
	
		\begin{theorem}\label{thm:2} 
		Let $3\leq\alpha<\omega$.
		\begin{itemize}
			\item[(i)] $\RA_{\alpha}^{csp}$ is finitely axiomatized over $\RA_{\alpha}^{cp}$
			by the equations (F0-F9).
			\item[(ii)] $\RA_{\alpha}^{cp}$ is not finitely axiomatizable over $\RA_{\alpha}^{p}$.
		\end{itemize} 
	\end{theorem}
	
		\begin{proof}[of Theorem \ref{thm:2}] Theorem \ref{thm:2}(i) is an immediate consequence of Theorem \ref{thm:1a}.  For the proof of (ii), note that Johnson \cite{J69} proved that the class of representable polyadic algebras is not finitely axiomatizable. I.e., the set of equations true in $\RA_{\alpha}^{csp}$ do not have a finite equational axiom set. Now, assume that the equational theory of $\RA_{\alpha}^{cp}$ has a finite equational axiom set $\Sigma$.  Then, $\Sigma$ together with all instances of (F0-F9) would be a finite  axiom set for $\RA_{\alpha}^{csp}$, by our Theorem \ref{thm:1a}. This contradicts Johnson's result.
			\end{proof}

\section{Ties to logic}
\def\Fm{\mathfrak{Fm}}

In this section we discuss the logical meaning of Theorem \ref{thm:1a} and its corollaries. 
Assume that $\alpha$ is a finite ordinal. Finite variable first-order logic has many versions depending on the number of variables, shape of atomic formulas and the logical connectives we use (see, e.g. \cite{AndrekaEtAl2022}). Here we consider the equality free, $\alpha$-variable first-order logic $L_{\alpha}$. It has infinitely many $\alpha$-place relation symbols $R_i$ for $i<\omega$, with the set of variables $V=\{v_0,\ldots,v_{\alpha-1}\}$, and the set of atomic formulas is
\[
	\big\{ R_k(v_{i_1}, \ldots, v_{i_{\alpha}}):\; k<\omega,\  i_1,\ldots,i_{\alpha}<\alpha  \big\},
\]
where inside the relation symbols we allow any order of the variables.
Formulas are built up from atomic formulas in the usual manner by using the logical connectives $\lnot$, $\exists v_i$, and $\land$. The derived connectives, such as $\lor$, $\to$, $\forall v_i$ are as usual.
The set of formulas is the universe of the formula algebra $\Fm^{c}$, where, for traditional reasons, we denote the operations defined by the connectives as $\cdot, -, \c_i$. 

Let $\p_{ij}$ denote the function on $\Fm^{c}$ that interchanges the variables $v_i$ and $v_j$ in each formula, and let $\Fm^{cp}$ denote the algebra $\Fm^{c}$ endowed with the functions $\p_{ij}$ for $i,j<\alpha$. 
Notice that the $\p_{ij}$'s are not connectives. 
Let $\equiv$ denote the tautological congruence of $\Fm^{c}$, that is, 
$\phi\equiv\psi$ iff the formula $\phi\leftrightarrow\psi$ is valid. It is not hard to see that $\equiv$ is a congruence relation with respect to the transposition functions $\p_{ij}$ as well, i.e. $\phi\equiv\psi$ implies that $\p_{ij}(\phi)\equiv\p_{ij}(\psi)$.\footnote{See \cite{AndrekaEtAl2022}, and \cite[Sec.4.3]{HenkinMonkTarski1971}.}

\def\Ps{\mathsf{Ps}}
\def\Dfs{\mathsf{Dfs}}

Consider the following classes of set algebras.
\begin{align*}
	\Dfs_{\alpha} &= \mathbf{SI}\big\{ \< \Sb({}^{\alpha}U), \cap, -, \C_i^U\>_{i<\alpha}:\; U\text{ is a set} \big\},\\ 
	\Ps_{\alpha} &= \mathbf{SI}\big\{ \< \Sb({}^{\alpha}U), \cap, -, \C_i^U, \P_{ij}^U\>_{i,j<\alpha}:\; U\text{ is a set} \big\} = \RA_{\alpha}^{cp}.
\end{align*}
$\Ps_{\alpha}$ is called the class of \emph{polyadic set algebras of dimension $\alpha$}, while $\Dfs_{\alpha}$ is its transposition-free subreduct, called \emph{diagonal-free cylindric set algebras}.
The following statement is known in algebraic logic \cite{AndrekaEtAl2022}, \cite[Sec.4.3]{HenkinMonkTarski1971}.

\begin{proposition}
	An equation is true in $\Fm^{c}/\!{}_\equiv$ if and only if it is true in $\Dfs_{\alpha}$.
\end{proposition}

The analogous statement for $\Fm^{cp}/\!{}_\equiv$ also holds, but as
$\Fm^{cp}/\!{}_\equiv$ is not a Lindenbaum--Tarski algebra (the operations do not come from connectives), we state it separately.

\begin{proposition}
	An equation is true in $\Fm^{cp}/\!{}_\equiv$ if and only if it is true in $\Ps_{\alpha}$.
\end{proposition}
\begin{proof}
	See Proposition 1 in \cite{ANTnew} with the modification that we do not have to deal with the equality formulas $v_i=v_j$. 
\end{proof}

Substitution of variables $\s_{ij}$ can further be added to $\Fm^{cp}$ as unary operations. However, one should be careful. If one takes the operation that replaces all occurrences of the variable $v_i$ with $v_j$ in a formula, then $\equiv$ no longer remains a congruence of this expanded algebra. For example, for $i\ne j$ we have
\begin{align*}
	&R(v_i)\equiv \exists v_jR(v_i), 
	\quad\text{ while }\\
	& R(v_j)\not\equiv \exists v_jR(v_j).
\end{align*}
This could happen because we replaced a free occurrence of $v_i$ with a bound occurrence of $v_j$ in $\exists v_jR(v_i)$, such substitutions are called \emph{non-admissible}. This is a well-known problem in first-order logic and the solution is that we choose a variable $v_k$ that does not occur in the formula and we change all bound occurrences of $v_j$ to $v_k$. This makes replacing $v_i$ with $v_j$ admissible while it does not affect the meaning of the formula. This kind of substitution, i.e., replacing together with renaming bound variables, is already ``semantic" in the sense that $\equiv$ is a congruence to it.
In ordinary first-order logic where we have infinitely many variables, this always can be done, but in finite-variable logic for a formula we may not find a new variable $v_k$. 

A semantics-preserving substitution for finite variable logic is introduced in \cite[p.354]{M71}. For our case of  replacing $v_i$ with $v_j$, Monk's definition amounts to the following. We replace free occurrences of $v_i$ to $v_j$ and at the same time we replace bound occurrences of $v_j$ to $v_i$. In more detail:
Let $\varphi$ be a formula. We get $[v_i/v_j]\varphi$ as follows. We go through the formula from left to right. If we encounter a free occurrrence of $v_i$ then we change it to $v_j$. If we encounter a bound occurrence of $v_j$ then we change it to $v_i$. It is proved in \cite[Lemma 1]{M71} that in each model and evaluation $s$ of variables in that model, $s$ satisfies $[v_i/v_j]\varphi$  if and only if $s(i/s_j)$ satisfies $\varphi$ in that model.

Let thus $\s_{ij}$ be the function on $\Fm^{c}$ that assigns $[v_i/v_j]\varphi$ to a formula $\varphi$.
Let $\Fm^{csp}$ denote the algebra expanded with these substitution operations $\s_{ij}$ for $i,j<\alpha$. Then $\equiv$ is a congruence relation on this algebra.
The class of algebras corresponding to $\Fm^{csp}/\!{}_\equiv$ is
$\RA_{\alpha}^{csp}$:
\begin{proposition}
	An equation is true in $\Fm^{csp}/\!{}_\equiv$ if and only if it is true in $\RA_{\alpha}^{csp}$.
\end{proposition}
\begin{proof}
	The proof is analogous to that of \cite[Proposition 1]{ANTnew}, except that we do not have to deal with the equality formulas $v_i=v_j$.
\end{proof}


\begin{thebibliography}{10}
\providecommand{\selectlanguage}[1]{\relax}

\bibitem{AndrekaEtAl2022}
H.~Andr\'eka, Z.~Gyenis, I.~N\'emeti, I.~Sain, \textbf{Universal Algebraic
  Logic: Dedicated to the Unity of Science}, 1st ed., Studies in Universal
  Logic, Birkh\"auser, Cham (2022).

\bibitem{algebraic-logic1991}
H.~Andr{\'e}ka, J.~D. Monk, I.~N{\'e}meti (eds.), \textbf{Algebraic Logic},
  North-Holland (1991).

\bibitem{ANTnew}
H.~Andr\'eka, I.~N\'emeti, Z.~Tuza, \emph{Transposition of variables is hard to
  describe}, \textbf{Submitted}, arXiv.2409.04088, (2024).

\bibitem{AAPALII97ec}
H.~Andréka, \emph{Complexity of equations valid in algebras of relations part
  II: Finite axiomatizations}, \textbf{Annals of Pure and Applied Logic},
  vol.~89(2) (1997), pp. 211--229.

\bibitem{HenkinMonkTarski1971}
L.~Henkin, J.~D. Monk, A.~Tarski, \textbf{Cylindric Algebras I--II}, vol. 64,
  115 of Studies in Logic and the Foundations of Mathematics, North-Holland
  Publishing Co., Amsterdam (1971).

\bibitem{J69}
J.~S. Johnson, \emph{Nonfinitizability of classes of representable polyadic
  algebras}, \textbf{Journal of Symbolic Logic}, vol.~34 (1969), pp. 344--352.

\bibitem{M69}
J.~D. Monk, \emph{Nonfinitizability of classes of representable cylindric
  algebras}, \textbf{Journal of Symbolic Logic}, vol.~34 (1969), pp. 331--343.
  
\bibitem{M71}
J.~D. Monk, \emph{Provability with finitely many variables}, 
\textbf{Proceedings of the American Mathematical Society}, vol.~27 (1971), pp. 353--358.
  
\bibitem{M99}
J.~D. Monk, \emph{Letter to Andr\'eka and N\'emeti}, 1997. 

\bibitem{P1} C. Pinter, \emph{Cylindric algebras and algebras of substitutions},
\textbf{Transactions of the American Mathematical Society}, vol.~175 (1973), pp. 167--179.

\bibitem{P2} C. Pinter, \emph{A simple algebra of first order logic},
\textbf{Notre Dame Journal of Formal Logic}, vol.~14 (1973), pp. 361--366.

\bibitem{SN96}
G.~S\'agi, \emph{A note on algebras of substitutions}, \textbf{Studia Logica},
vol.~72 (2002), pp. 265--284.


\bibitem{sain-thompson1991}
I.~Sain, R.~J. Thompson, \emph{Strictly finite schema axiomatization of
  quasi-polyadic algebras}, [in:] Andr{\'e}ka \emph{et~al.}
  \cite{algebraic-logic1991}, pp. 539--571.

\end{thebibliography}
\end{document}